\documentclass[12pt]{amsart}
\usepackage{graphicx}
\usepackage{amsfonts,amssymb,amscd,amsmath,enumerate,verbatim,calc}
\usepackage{enumitem}
\usepackage{mathtools}
\usepackage{hyperref}

\newtheorem{Theorem}{Theorem}[section]
\newtheorem{Lemma}[Theorem]{Lemma}

\newtheorem{Corollary}[Theorem]{Corollary}
\newtheorem{Proposition}[Theorem]{Proposition}
\newtheorem{Remark}[Theorem]{Remark}

\begin{document}
\title{Distinguishing Number of Non-Zero Component Graphs}
\author{I. Javaid$^*$, M. Murtaza, H. Benish}
\keywords{Automorphism, distinguishing labeling, distinguishing number. \\
\indent 2010 {\it Mathematics Subject Classification.} 05C25\\
\indent $^*$ Corresponding author: ijavaidbzu@gmail.com}
\address{Centre for advanced studies in Pure and Applied Mathematics,
Bahauddin Zakariya University Multan, Pakistan\newline Email:
ijavaidbzu@gmail.com, mahru830@gmail.com, hira\_benish@yahoo.com.}


\date{}
\maketitle

\begin{abstract}
A non-zero component graph $G(\mathbb{V})$ associated to a finite vector space $\mathbb{V}$ is a graph whose vertices are non-zero vectors of $\mathbb{V}$ and two vertices are adjacent, if their corresponding vectors have at least one non-zero component common in their linear combination of basis vectors. In this paper, we extend the study of properties of automorphisms of non-zero component graphs. We prove that every permutation of basis vectors can be extended to an automorphism of $G(\mathbb{V})$. We prove that the symmetric group of basis vectors of $\mathbb{V}$ is isomorphic to the automorphism group of $G(\mathbb{V})$. We find the distinguishing number of the graph for both of the cases, when the number of field elements of vector space $\mathbb{V}$ are 2 or more than 2.
\end{abstract}

\section{Preliminaries}
The association of graphs with algebraic structures has become an interesting
research topic for the past few decades. See for instance: the study
of zero-divisor graphs of commutative rings with unity was initiated by Beck \cite{Beck} to discuss coloring problem. Commuting graphs associated to symmetric groups were studied by Bondy \emph{et al.} \cite{bates1,bates2}, where the authors  discussed connectivity and related properties of these graphs. Power graphs for groups and semigroups were discussed in
\cite{Cameron,Chakrabarty,Moghaddamfar}. Intersection graphs associated to vector spaces were studied in \cite{Rad,Talebi}.
Das assigned non-zero component graphs to finite dimensional vector
spaces in \cite{Das}. The author also studied its domination number and independence
number. In \cite{Das1}, the author discussed edge-connectivity and
the chromatic number of the graph. Non-zero component graphs have very interesting symmetrical structures especially for the case where the number of field elements of vector space is 2. Several authors studied non-zero component graphs for different graph parameters. For instance, metric dimension and partition
dimension of non-zero component graphs are studied in \cite{us}.
Fixing number of the graph is studied in \cite{fa}. Locating-dominating sets and identifying codes of non-zero
component graphs are discussed by Murtaza {\it et al.} \cite{Murtaza}. Automorphisms of non-zero component graphs are studied by Hira {\it et al.} in \cite{HB1}, where the authors studied the fixing neighborhoods of pairs of vertices of the graph and the fixed number of the graph. In this paper, we extend the study of properties of automorphisms non-zero component graphs. We find the cardinality of the automorphism group of the graph, distinguishing number of the graph and prove that the automorphism group of the graph is destroyed by only 2 colors for the case where the number of field elements are 2.

How many colors are needed to identify apparently same keys of a key ring? The first time this question was discussed by Rubin \cite{Rubin} in 1980. The problem was transformed into other problems like scheduling meetings, storing chemicals and solution to these problems need proper coloring and one with a small number of colors. Motivated by this problem, Albertson, and Collins
\cite{Alb} introduced the concept of the distinguishing number of a graph as follows: A labeling $f: V(G)\rightarrow \{1,2,3,...,t\}$ is
called a $t$-\emph{distinguishing} if no non-trivial automorphism of
a graph $G$ preserves the vertex labels. The \emph{distinguishing
number} of a graph $G$, denoted by $Dist(G)$, is the least integer
$t$ such that $G$ has $t$-distinguishing labeling.
The distinguishing number of a complete graph $K_n$ is $n$, the
distinguishing number of a path graph $P_n$ is $2$ and the distinguishing
number of a cyclic $C_{n},\ n\geq 6$ is $2$. For a graph $G$ of order
$n$, $1\le Dist(G)\le n$ \cite{Alb}.

Now, we define some graph related terminology which is used in this
article: Let $G$ be a graph with the vertex set $V(G)$ and the edge
set $E(G)$. Two vertices $u$ and $v$ are \emph{adjacent}, if they
share an edge, otherwise they are called \emph{non-adjacent}. The
number of adjacent vertices of a vertex $v$ is called the \emph{degree} of $v$ in
$G$. For a graph $G$, an {\it automorphism} of $G$ is a bijective
mapping $f$ on $V(G)$ such that $f(u)f(v)\in E(G)$ if and only if
$uv\in E(G)$. The set of all automorphisms of $G$ forms a group,
denoted by $\Gamma(G)$, under the operation of composition. For a
vertex $v$ of $G$, the set $\{f(v):f\in \Gamma(G)\}$ is the {\it
orbit} of $v$, denoted by $\mathcal{O}(v)$. If two vertices $u,v$
are mapped on each other under the action of an automorphism $g\in
\Gamma(G)$, then we write it $u\sim^gv$. If $u,v$ cannot be mapped on
each other, then we write $u\nsim^gv$. An automorphism $g\in
\Gamma(G)$ is said to \emph{fix} a vertex $v\in V(G)$ if $v\sim^g
v$. The {\it stabilizer} of a vertex $v$ is the set of all
automorphisms that fix $v$ and it is denoted by $\Gamma_v(G)$. Also,
$\Gamma_v(G)$ is a subgroup of $\Gamma(G)$. Let us consider sets
$S(G)=\{v\in V(G):$ $|\mathcal{O}(v)|\ge 2 \}$ and
$V_s(G)=\{(u,v)\in S(G)\times S(G):$ $u\neq v$ and
$\mathcal{O}(u)=\mathcal{O}(v)\}$. If $G$ is a rigid graph (i.e., a
graph with $\Gamma (G)={id}$), then $V_s(G)=\emptyset$. For $v\in
V(G)$, the subgroup $\Gamma_v(G)$ has a natural action on $V(G)$ and
the orbit of $u$ under this action is denoted by $\mathcal{O}_v(u)$
i.e., $\mathcal{O}_v(u)=\{g(u):g\in\Gamma_v(G)\}$. An automorphism
$g\in \Gamma(G)$ is said to {\it fix} a set $D\subseteq V(G)$ if for
all $v\in D$, $v\sim^g v$. The set of automorphisms that fix $D$,
denoted by $\Gamma_D(G)$, is a subgroup of $\Gamma(G)$ and
$\Gamma_D(G)=\cap_{v\in D}\Gamma_v(G)$.

Throughout the paper, $\mathbb{V}$ denotes a vector space of
dimension $n$ over the field of $q$ elements and
$\{b_1,b_2,...,b_n\}$ be a basis of $\mathbb{V}$. The \emph{non-zero
component graph} of $\mathbb{V}$ \cite{Das}, denoted by $G(\mathbb{V})$, is a
graph whose vertex set consists of non-zero vectors of
$\mathbb{V}$ and two vertices are joined by an edge if they share at
least one $b_{i}$ with non-zero coefficient in their unique linear
combination with respect to basis vectors $\{b_{1}, b_{2},\ldots,b_{n}\}$. It is proved in \cite{Das} that $G(\mathbb{V})$ is
independent of the choice of basis, i.e., isomorphic non-zero
component graphs are obtained for two different bases. In
\cite{Das}, Das studied automorphisms of $G(\mathbb{V})$. It was
shown that an automorphism maps basis of $G(\mathbb{V})$ to a basis
of a special type, namely non-zero scalar multiples of a permutation
of basis vectors.
\begin{Theorem}\cite{Das}
Let $\varphi:G(\mathbb{V})\rightarrow G(\mathbb{V})$ be a graph
automorphism. Then, $\varphi$ maps a basis $\{\alpha_1,
\alpha_2,...,\alpha_n\}$ of $\mathbb{V}$ to another basis $\{\beta_1,
\beta_2,...,\beta_n\}$ such that there exists a permutation $\sigma$
from the symmetric group on $n$ elements, where each $\beta_i$ is of the
form $c_i\alpha_{\sigma(i)}$ and each $c_i$'s are non-zero.
\end{Theorem}

The \emph{skeleton} of a vertex $u\in V(G(\mathbb{V}))$ denoted by
$S_u$, is the set of all those basis vectors of $\mathbb{V}$ which
have non-zero coefficients in the representation of $u$ as the
linear combination of basis vectors. In \cite{Murtaza}, we partition
the vertex set of $G(\mathbb{V})$ into $n$ classes $T_i$, $(1\le i
\le n)$, where $T_i=\{v\in \mathbb{V}: |S_v|=i\}$. For example, if
$n=4$ and $q=2$, then $T_3=\{b_1+b_2+b_3,b_1+b_2+b_4,b_1+b_3+b_4,b_2+b_3+b_4\}$.

The section wise break up of the article is as follows: In Section \ref{ANZG}, we study the properties of automorphisms of non-zero component graph. We discuss the relation between vertices and their images under an automorphism of $G(\mathbb{V})$ in the terms of their skeletons. We prove that every permutation of basis vectors can be extended to an automorphism of $G(\mathbb{V})$. We also prove that the symmetric group of basis vectors is isomorphic to the automorphism group of $G(\mathbb{V})$. We find the cardinality of automorphism group of $G(\mathbb{V})$. In Section \ref{DLNG}, we study distinguishing labelings and distinguishing number of non-zero component graphs for both of the cases, when the number of field elements of vector space $\mathbb{V}$ are 2 or more than 2.

\section{Automorphisms of non-zero component graph}\label{ANZG}
In \cite{HB1}, we study the properties of automorphisms of $G(\mathbb{V})$ of a vector space $\mathbb{V}$ of dimension $n\geq 3$ over the field of $2$ elements. Graphs $G(\mathbb{V})$ have more interesting symmetrical structures for the case where the number of field elements are 2, as compared to the cases where the number of field elements are more than 2, as $G(\mathbb{V})$ does not have twin vertices in the former case. In this section, we extend the study of the properties of automorphisms of $G(\mathbb{V})$, where $\mathbb{V}$ is a vector space of dimension $n$ over the field of 2 elements.

\begin{Lemma}\label{Murtaza}\cite{Murtaza}
If $v\in T_s$ for $s$ $(1\le s \le n)$, then
$deg(v)=(2^s-1)2^{n-s}-1$.
\end{Lemma}

\begin{Lemma}\label{VertexDeg}\cite{HB1}
Let $u,v\in V(G(\mathbb{V}))$ such that $u\in T_r$ and $v\in T_s$
where $r\ne s$ and $1\le r,s \le n$, then $u\not\sim^g v$ for all
$g\in \Gamma(G(\mathbb{V}))$.
\end{Lemma}

From Lemma \ref{VertexDeg}, we have the following straightforward
remark.

\begin{Remark}\cite{HB1}
Since $u\in T_n$ where $S_u=\{b_1,b_2,...,b_n\}$ is the only element in
$T_n$. Therefore, by Lemma \ref{VertexDeg}, $g(u)=u$ for all $g\in \Gamma(G(\mathbb{V}))$.
\end{Remark}

\begin{Lemma}\cite{HB1}
Let $b_l\in T_1$ be a basis vector and $g\in \Gamma_{b_l}$. Let
$u\in V(G(\mathbb{V}))$, then $b_l\in S_u$ if and only if $b_l\in
S_{g(u)}$.
\end{Lemma}

\begin{Lemma}\label{skelotondiff}\cite{HB1}
Let $u,v\in T_i$ for some $i$ $(1\le i \le n-1)$ and $u\sim^g v$ for
some $g\in \Gamma(G(\mathbb{V}))$. The following statements hold:
\begin{enumerate}[label=(\roman*)]
  \item If $b\in S_u \cap S_v$, then $g(b)\in S_u \cap S_v$.
  \item If $b\in S_u-S_v$, then $g(b)\in S_v-S_u$.
\end{enumerate}
\end{Lemma}

\begin{Lemma}\label{LemTransOnlyNghbr}\cite{HB1}
Let $b_l,b_m\in T_1$ be two distinct basis vectors of $\mathbb{V}$
and $g\in \Gamma(G(\mathbb{V}))$ be an automorphism such that
$b_l\sim^g b_m$. Let $u\in V(G(\mathbb{V}))$, then we have:
\begin{enumerate}[label=(\roman*)]
  \item If $b_l\in S_u$ and $b_m\not\in S_u$, then
$b_l\not\in S_{g(u)}$ and $b_m\in S_{g(u)}$.
  \item $b_l,b_m\in S_u$ if and only if
$b_l,b_m \in S_{g(u)}$.
\end{enumerate}
\end{Lemma}

\begin{Lemma}\label{BasisDistinctSkelton}
   Let $b_l,b_m\in T_1$ be two distinct basis vectors, then there exist ${{n-1}\choose{i-1}}-{{n-2}\choose{i-2}}$ pairs of distinct vertices $u,v\in T_i$ for each $i$ $(1\le i \le n-1)$, such that $b_l\in S_u\setminus S_v$ and $b_m\in S_v\setminus S_u$.
\end{Lemma}

\begin{proof}
For $i=1$, $b_l\in S_{b_l}\setminus S_{b_m}$ and $b_m\in S_{b_m}\setminus S_{b_l}$. For $2\le i\le n-1$, there are $n\choose i$ vertices in each $T_i$ and out of these $n\choose i$ vectors, there are ${n-1}\choose{i-1}$ vectors that contain $b_l$ in their skeletons. Similarly, there are ${n-1}\choose{i-1}$ vectors in each $T_i$ that contain $b_m$ in their skeletons. Also, there are ${n-2}\choose{i-2}$ vectors that contain both $b_l$ and $b_m$ in their skeletons. Thus, there are ${{n-1}\choose{i-1}}-{{n-2}\choose{i-2}}$ vectors in each $T_i$ that contain $b_l$ and but not $b_m$ in their skeletons.
\end{proof}

The following Lemma shows that every non-trivial automorphism $g$ of $G(\mathbb{V})$ does not belongs to the stabilizer of at least two of the basis vectors, i.e., $g$ moves at least two basis vectors from their places.

\begin{Lemma}\label{AutoHasBasis}
   Let $g\in \Gamma(G(\mathbb{V}))$ be a non-trivial automorphism, there exist at least two distinct basis vectors $b_l,b_m\in T_1$, such that $g(b_l)=b_m$.
\end{Lemma}

\begin{proof}
  Since $g$ is non-trivial, therefore there exist distinct vectors $u,v\in V(G(\mathbb{V}))$, such that $g(u)=v$. Also, by the contrapositive argument of Lemma \ref{VertexDeg}, both $u$ and $v$ belong to the same class $T_i$ for some $i$, $1\le i \le n-1$. If $i=1$, then the result is obvious. Therefore, we assume $2\le i \le n-1$. Since $u$ and $v$ are distinct, therefore $S_u-S_v$ is non-empty. Let $b_l\in S_u-S_v$, then $b_l$ is adjacent to $u$ and non-adjacent to $v$. Since $g$ is an automorphism, therefore $g(b_l)$ is  adjacent $g(u)=v$ but non-adjacent to $g(v)$. Thus, $g(b_l)\in S_v-S_{g(v)}$. Clearly, $b_l$ and $g(b_l)=b_m$ (say) are distinct basis vectors because if $g(b_l)=b_l$, then $b_l\in S_v$, a contradiction. Hence, proved.
\end{proof}

 The next result shows that every permutation of basis vectors can be extended to an automorphism of non-zero component graph.
\begin{Theorem}\label{permbasis}
  Every permutation of basis vectors of $G(\mathbb{V})$ can be extended to an automorphism of $G(\mathbb{V})$.
\end{Theorem}

\begin{proof}
Let $h$ be a permutation of basis vectors. Let $u\in V(G(\mathbb{V}))$ be an arbitrary vertex, then $u\in T_m$ for some $m$, $1\le m \le n$. Let $S_u=\{b_1,b_2,...,b_m\}$. We define an extension $g: V(G(\mathbb{V}))\rightarrow  V(G(\mathbb{V}))$ of $h$ on vertex set $V(G(\mathbb{V}))$ as $g(u)=v$, where $v$ is a vector whose skeleton is $S_v=\{h(b_1),h(b_2),...,h(b_m)\}$. We claim that $g$ is an automorphism of $G(\mathbb{V})$. If $u$ is a basis vector, then $g(u)=h(u)$. Let $u_1,u_2\in T_m$ be two distinct vectors, then $S_{u_1}\ne S_{u_2}$. Since $h$ is a permutation on basis vectors, therefore $S_{g(u_1)}\ne S_{g(u_2)}\Rightarrow g(u_1)\ne g(u_2)$. Now, we show that $g$ preserves the relation of adjacency. Let $u,v,w\in V(G(\mathbb{V}))$ be any three vertices such that $u$ is adjacent $v$ and non-adjacent to $w$. Then $S_u\cap S_v\ne \emptyset$ and $S_u\cap S_w = \emptyset$. Since $h$ is a permutation, therefore disjoint sets of basis vectors have disjoint image sets and overlapping sets of basis vectors have overlapping image sets under permutation $h$. Thus, $S_{g(u)}\cap S_{g(v)}\ne \emptyset$ and $S_{g(u)}\cap S_{g(w)} = \emptyset$ and hence, $g(u)$ is adjacent to $g(v)$ and $g(u)$ is non-adjacent to $g(w)$. Hence, $g$ preserves the relation of adjacency and non-adjacency among the vertices of $G(\mathbb{V})$.
\end{proof}
\begin{Corollary}\label{AutoPerm}
  Every automorphism of $G(\mathbb{V})$ can be restricted to a permutation of basis vectors.
\end{Corollary}

\begin{proof}
  Let $g\in Aut(G(\mathbb{V}))$ be an arbitrary automorphism. Since every automorphism is itself a permutation on the vertex set of $G(\mathbb{V})$, therefore the restriction of $g$ on the elements of $T_1$ forms a permutation $h$ of basis vectors.
\end{proof}

Let $Sym$ denotes the symmetric group basis vectors $\{b_1,b_2,...,b_n\}$, then the following theorem shows that the symmetric group of basis vectors is isomorphic to the automorphism group of the graph.

\begin{Theorem}\label{PermAutoBijection}
  The symmetric group of basis vectors $\{b_1,b_2,...,b_n\}$ is isomorphic to the automorphism group of $G(\mathbb{V})$.
  $$ Sym \cong Aut(G(\mathbb{V})) $$.
\end{Theorem}

\begin{proof}
   Let $h\in Sym$ be a permutation of basis vectors. Also let $g\in Aut(G(\mathbb{V}))$. We define a function $\psi: Sym\rightarrow Aut(G(\mathbb{V}))$ as $\psi(h)=g$, where $g$ is the extension of $h$ as defined in Theorem \ref{permbasis}, i.e., action of $g$ on a vertex $u$ is defined by the action of permutation $h$ on the elements of basis vectors in the skeleton of $u$. We claim that $\psi$ is a bijective function.
   \begin{enumerate}
     \item \textbf{$\psi$ is well defined}: Let $g_1,g_2\in Aut(G(\mathbb{V}))$ such that $g_1\ne g_2$. Let $h_1,h_2\in Sym$ be two permutations of basis vectors such that $g_1$ is the extension of $h_1$ and $g_2$ is the extension of $h_2$. Since $g_1\ne g_2$, therefore there exist a vertex $u\in T_m$ for some $m$, $1\le m \le n-1$ such that $g_1(u)\ne g_2(u)$. Let $S_u=\{b_1,b_2,...,b_m\}$. Then $S_{g_1(u)}\ne S_{g_2(u)}$ $\Rightarrow$ $\{h_1(b_1),h_1(b_2),...,h_1(b_m)\}\ne$  $\{h_2(b_1),h_2(b_2),...,h_2(b_m)\}$ $\Rightarrow$ $h_1(b_j)\ne h_2(b_j)$ for at least one $j$, $(1\le j \le m)$ $\Rightarrow$ $h_1\ne h_2$.
     \item \textbf{$\psi$ is one-one}: Let $h_1,h_2\in Sym$ be two permutations such that $h_1\ne h_2$. Let $g_1$ and $g_2$ be the extensions of $h_1$ and $h_2$, respectively. Since $h_1\ne h_2$, therefore $h_1(b_j)\ne h_2(b_j)$ for at least one $j$, $(1\le j \le m)$ $\Rightarrow$ $\{h_1(b_1),h_1(b_2),...,h_1(b_m)\}\ne$  $\{h_2(b_1),h_2(b_2),...,h_2(b_m)\}$ $\Rightarrow$ $S_{g_1(u)}\ne S_{g_2(u)}$ $\Rightarrow$ $g_1(u)\ne g_2(u)$ $\Rightarrow$ $g_1\ne g_2$.
     \item \textbf{$\psi$ is onto}: By Corollary \ref{AutoPerm} the restriction of every automorphism $g\in Aut(G(\mathbb{V}))$ forms a permutation $h$ of basis vectors. Therefore, $g$ is an extension of permutation $h$ of basis vectors. Thus, for an automorphism $g\in Aut(G(\mathbb{V}))$, there exist a permutation $h\in Sym$ such that $\psi(h)=g$.
     \item \textbf{$\psi$ is homomorphism}: Let $h_1,h_2\in Sym$ and $g_1,g_2\in Aut(G(\mathbb{V}))$ such that $g_1$ and $g_2$ are the extension automorphisms of $h_1$ and $h_2$, respectively, as defined in Theorem \ref{permbasis}. We claim that $\psi(h_1h_2)=$ $\psi(h_1)\psi(h_2)=g_1g_2$. Let $\psi(h_1h_2)=g$ and $u\in T_m$ for some $m$, $1\le m \le n-1$ such that $S_u=\{b_1,b_2,...,b_m\}$. Then $g(u)$ has skeleton $S_{g(u)}=\{h_1h_2(b_1),h_1h_2(b_2),...,h_1h_2(b_m)\}$ $=\{h_1(h_2(b_1)),$ $h_1(h_2(b_2)),...,h_1(h_2(b_m))\}=$ $S_{g_1(g_2(u))}=$ $S_{g_1g_2(u)}$ $\Rightarrow$ $g=g_1g_2$
         $\Rightarrow$ $\psi(h_1h_2)=$ $\psi(h_1)\psi(h_2)$.
   \end{enumerate}
Thus, $\psi$ is an isomorphism between $Sym$ and $Aut(G(\mathbb{V}))$.
\end{proof}

Theorem \ref{PermAutoBijection} leads to the following straight forward result.
\begin{Theorem}
  Non-zero component graph $G(\mathbb{V})$ has $n!$ automorphisms. $$|Aut(G(\mathbb{V}))|=n!$$
\end{Theorem}

\begin{proof}
Proof simply follows from the fact that the symmetric group of basis vectors has $n!$ elements and the symmetric group of basis vectors is isomorphic to the automorphism group of the graph.
\end{proof}

\section{Distinguishing labeling of non-zero component graph}\label{DLNG}

We have seen in the previous section that the symmetric group of basis vectors is isomorphic to the automorphism group of non-zero component graph of a vector space of dimension $n\ge 3$ over the field of 2 element. Therefore, it is sufficient to destroy the symmetric group of basis vectors, in order to destroy the automorphism group of $G(\mathbb{V})$. In this section, we again consider the case where $\mathbb{V}$ is a vector space over the field of 2 elements except for Theorem \ref{distfield} where field elements are more than 2. We use distinguishing labeling with the minimum number of labels (or colors) to destroy the symmetric group of basis vectors which leads us to destroy the automorphism group of the graph.

\begin{Lemma}\label{LemDiffLabels}
Let $b_l,b_m\in T_1$ be two distinct basis vectors of $\mathbb{V}$
and $u,v\in T_i$ for some $i$ $(2\le i \le n-1)$ such that $b_l\in S_u\setminus S_v$ and $b_m\in S_v\setminus S_u$. Let $f: V(G(\mathbb{V}))\rightarrow \{1,2\}$ be a labeling such that $f(u)\ne f(v)$, then $f$ breaks all automorphisms $g\in Aut(G(\mathbb{V}))$ which maps $b_l$ and $b_m$ on each other, i.e., $b_l\not\sim b_m$.
\end{Lemma}

\begin{proof}
  If $f(u)\ne f(v)$, then $u$ and $v$ have different labels, therefore these cannot map on each other and labeling $f$ breaks all the automorphisms that map $u$ and $v$ on each other. Since $b_l$ is adjacent to $u$ but non-adjacent to $v$ and $b_m$ is adjacent to $v$ but non-adjacent to $u$ and $u$ and $v$ cannot map on each other, therefore by Lemma \ref{LemTransOnlyNghbr}(i), $b_l$ and $b_m$ cannot map on each other, as the automorphisms in $Aut(G(\mathbb{V}))$ that maps $b_l$ and $b_m$ are broken by labeling $f$. Hence, $b_l \not\sim b_m$.
\end{proof}

Since every permutation of basis vectors can be written as the product of transpositions of basis vectors, therefore we label the vertices of $G(\mathbb{V})$ with labeling $f$ in such a way that all those automorphisms of $G(\mathbb{V})$ are destroyed that contain transpositions of basis vectors.

\begin{Theorem}
  Let $G(\mathbb{V})$ be the non-zero component graph of a vector space of dimension $n\ge 3$ over the field of 2 elements. Then $Dist(G(\mathbb{V}))=2$.
\end{Theorem}

\begin{proof}
Since $G(\mathbb{V})$ is not a rigid graph, therefore $Dist(G(\mathbb{V}))\ge 2$. We define labeling $f:V(G(\mathbb{V}))\rightarrow \{1,2\}$ in such a way that after assigning the labels to the vertices of $G(\mathbb{V})$, all those automorphisms of $G(\mathbb{V})$ are destroyed that contain transpositions of basis vectors. There are $n\choose 2$ transpositions of basis vectors. We proceed by assigning labels to the vertices of classes $T_1$, $T_{n-1},T_2$.

\begin{enumerate}
  \item \textbf{Labeling the vertices of $T_1$}: Consider labeling $f$ defined on the vertices of $T_1$, i.e., basis vectors as:
$$
f({b_i})=\left\{
\begin{array}{lcc}
    1 & if & 1\le i\le \lfloor {\frac{n}{2}}\rfloor \\
    2 & if & \lfloor {\frac{n}{2}}\rfloor+1 \le i\le n
\end{array}
\right.
$$

Since, basis vectors $b_i$, for all $i$ $(1\le i\le \lfloor {\frac{n}{2}}\rfloor)$ have different labels from the labels of $b_j$, for all $j$ $(\lfloor {\frac{n}{2}}\rfloor+1 \le j\le n)$, therefore $b_i\not\sim b_j$ for all $i,j$ $(i\ne j)$ where $(1\le i\le \lfloor {\frac{n}{2}}\rfloor)$ and $(\lfloor {\frac{n}{2}}\rfloor+1 \le j\le n)$ under all automorphisms of $Aut(G(\mathbb{V}))$. If $n$ is even, then ${n\choose 2}-2{{\frac{n}{2}}\choose 2}=$ $\frac{n^2}{4}$ transpositions of basis vectors are destroyed by labeling the vertices of $T_1$. If $n$ is odd, then ${n\choose 2}-{{\lfloor {\frac{n}{2}} \rfloor} \choose 2}-{{\lfloor {\frac{n}{2}} \rfloor+1} \choose 2} =$ $\frac{n^2-1}{4}$ transpositions of basis vectors are destroyed by labeling the vertices of $T_1$.
  \item \textbf{Labeling the vertices of $T_{n-1}$}: We label the vertices $u\in T_{n-1}$ as:
      $$
        f({u})=\left\{
            \begin{array}{lcl}
                1 & if & S_u=\{b_2,b_3,...,b_{\lfloor {\frac{n}{2}}\rfloor}\}\,\,\, or\,\,\, S_u=\{b_{\lfloor {\frac{n}{2}}\rfloor+2},b_{\lfloor {\frac{n}{2}}\rfloor+3},...,b_n\}\\
                2 & if & otherwise
            \end{array}
        \right.
        $$
        Consider $u,v\in T_{n-1}$ such that $S_u=\{b_2,b_3,...,b_{\lfloor {\frac{n}{2}}\rfloor}\}$ and $S_v=\{b_1,b_3,b_4,...,$ $b_{\lfloor {\frac{n}{2}}\rfloor}\}$, then $f(u)\ne f(v)$. Also, $b_2\in S_u\setminus S_v$ and $b_1\in S_v\setminus S_u$. Therefore, by Lemma \ref{LemDiffLabels} labeling $f$ breaks the automorphisms that map $b_1$ and $b_2$ on each other (reader can notice that, even basis vectors $b_1$ and $b_2$ have same label 1 assigned by labeling $f$, but the automorphism that maps $b_1$ and $b_2$ on each other is destroyed due to $f(u)\ne f(v)$). Thus, $b_1\not\sim b_2$. In the similar way, reader can verify $b_1\not\sim b_j$ where $3\le j\le \lfloor {\frac{n}{2}}\rfloor$ and $b_{\lfloor {\frac{n}{2}}\rfloor+1}\not\sim b_k$ where $\lfloor {\frac{n}{2}}\rfloor+2\le k\le n$. There are $n-2$ transposition of basis vectors that are destroyed by labeling the vertices of $T_{n-1}$.

  \item \textbf{Labeling the vertices of $T_2$}: We label the vertices $u\in T_{2}$ as:
      $$
        f({u})=\left\{
            \begin{array}{lcl}
                1 & if & S_u=\{b_i,b_{i+1}\}\,\,where\,\,1\le i \le n-1\\
                2 & if & otherwise
            \end{array}
        \right.
        $$
  We claim that labeling $f$ defined on the vertices of $T_2$ destroys the remaining transpositions of basis vectors. The remaining transpositions of basis vectors are $b_i\sim b_j$ where $i\ne j$ $(2\le i,j \le \lfloor \frac{n}{2} \rfloor)$ and $b_i\sim b_j$ where $i\ne j$ $(\lfloor \frac{n}{2} \rfloor+1 \le i,j\le n)$.
  Consider the first case and let $b_l\sim b_m$ for some $l\ne m$ $(2\le l,m \le \lfloor \frac{n}{2} \rfloor)$. Then there exist vertices $u,v\in T_2$ such that $S_{u}=\{b_{l-1},b_l\}$ and $S_v=\{b_{l-1}, b_m\}$. Then by the definition of labeling $f$, $f(u)=1\ne 2=f(v)$ and $b_l\in S_u\setminus S_v$ and $b_m\in S_v\setminus S_u$. Hence by Lemma \ref{LemDiffLabels}, labeling $f$ breaks the automorphism that maps $b_l$ and $b_m$ on each other. Similarly, for the second case, let $b_l\sim b_m$ for some $l\ne m$ $(\lfloor \frac{n}{2} \rfloor+1 \le l,m\le n)$ and using the same arguments as in the first case we see that labeling $f$ breaks the automorphism that map $b_l$ and $b_m$ on each other. Thus, all the remaining transpositions of basis vectors are destroyed by labeling $f$. If $n$ is even, then ${n\choose 2}-\frac{n^2}{4}-(n-2)=$ $\frac{n^2-6n+8}{4}$ transpositions of basis vectors are destroyed by labeling the vertices of $T_2$. Similarly, if $n$ is odd, then ${n\choose 2}-\frac{n^2-1}{4}-(n-2)=$ $\frac{n^2-6n+9}{4}$ transpositions of basis vectors are destroyed by labeling the vertices of $T_2$.
  \item \textbf{Labeling the vertices of $T_3$,...,$T_{n-2},T_n$}: The vertices in the remaining classes $T_3$,...,$T_{n-2},T_n$ can be labeled with any one label 1 or 2.
\end{enumerate}
We have seen that labeling $f$ destroys all possible transpositions of basis vectors and hence destroys the permutation group of basis vectors. Thus, by Theorem \ref{PermAutoBijection}, the automorphism group of $G(\mathbb{V})$ is destroyed
by labeling $f$. Hence, $f$ is a 2-distinguishing labeling of $G(\mathbb{V})$ and consequently, $Dist(G(\mathbb{V}))=2$.

\end{proof}

We now discuss the case when $G(\mathbb{V})$ is non-zero component graph of vector space $\mathbb{V}$ where $n\ge 3$ and $q\ge 3$. Since $G(\mathbb{V})$ has $n\choose i$ twin sets in classes $T_i$ for each $i$ $(1\le i \le n)$. Also each of these twin sets has $(q-1)^i$ number of vertices. Let $T_{i_k}$ $1\le k \le {n\choose i}$ denotes the $k$th twin set in class $T_{i}$. Since at least $m$ labels are required to label a twin set of cardinality $m$, therefore we have the following result for the distinguishing number of a graph which has twin sets.

\begin{Proposition}\label{distwin}\cite{fazil2018distinguishing}
  Let $W_1,W_2,...,W_t$ are disjoint twin sets of connected graph $G$ and $m=\max\{|W_i|:$ where $1\le i \le t\}$, then $Dist(G)\ge m$.
\end{Proposition}

\begin{Theorem}\label{distfield}
  Let $G_{\mathbb{V}}$ be the non-zero component graph of vector space $\mathbb{V}$ of dimension $n\ge 3$ and $q\ge 3$. Then $Dist(G(\mathbb{V}))=(q-1)^n$
\end{Theorem}

\begin{proof}
Since twin sets $T_{i_k}$ $(1\le i\le n)$ $(1\le k \le {n\choose i})$ are disjoint and each contains $(q-1)^i$ twin vertices. Also, $(q-1)^n=\max\{|T_{i_k}|:$ $(1\le i\le n)\}$. Therefore, by Proposition \ref{distwin}, $Dist(G(\mathbb{V}))\ge (q-1)^n$. Also, $(q-1)^n\ge (q-1)^i$ for all $i$ $(1\le i\le n)$, therefore twin sets $T_{i_k}$ are independently labeled by $(q-1)^i$ labels out of $(q-1)^n$ labels. Thus, all automorphisms of disjoint twin sets $T_{i_k}$ are destroyed. Hence, $Dist(G_{\mathbb{V}})= (q-1)^n$.
\end{proof}

\noindent\textbf{Acknowledgments}\\
This research was supported by Higher Education Commission of Pakistan with grant no. 7354/Punjab/NRPU/R\&D/HEC/2017.


\end{document}